\documentclass[runningheads]{llncs}
\usepackage[T1]{fontenc}
\usepackage{graphicx}
\usepackage{amsmath, amsfonts} 

\begin{document}
\title{Coherent states and entropy\thanks{Partially supported by a 2022 Bridge grant, University of Western Ontario, proposal ID 53917.}}
%
%
\author{Tatyana Barron\inst{1} \and
Alexander Kazachek\inst{2,3}}
\authorrunning{T. Barron and A. Kazachek}
%
\institute{Department of Mathematics, University of Western Ontario, London, Ontario, Canada N6A 5B7 
\email{tatyana.barron@uwo.ca}\\ \and 
Department of Mathematics, University of Western Ontario, London, Ontario, Canada N6A 5B7 
\email{akazache@uwo.ca}\\
\and
 Department of Applied Mathematics, University of Waterloo, Waterloo, Ontario, Canada N2L 3G1}
\maketitle              
\begin{abstract}
Let $H_k$, $k\in {\mathbb{N}}$, be the Hilbert spaces of geometric quantization on a K\"ahler manifold $M$.  
With two points in  $M$ we associate a Bell-type state $b_k \in H_k\otimes H_k$. When $M$ is compact or when $M$ is ${\mathbb{C}}^n$, we provide positive lower bounds for the entanglement entropy of $b_k$ (asymptotic in $k$, as $k\to\infty$).   

\end{abstract}

\section{Introduction and main results} Let $H_k$, $k=1,2,3,...$ be the Hilbert spaces of K\"ahler quantization on a K\"ahler manifold $M$. Let $p,q\in M$ and 
$\Theta_p^{(k)}\in H_k$ and $\Theta_q^{(k)}\in H_k$ be the coherent states at $p$ and $q$ respectively. With the pair $p,q$ we associate the Bell-type pure state 
\begin{equation}
\label{statew}
w_k=w_k(p,q)=\frac{1}{||\Theta_p^{(k)}||^2}\Theta_p^{(k)}\otimes \Theta_p^{(k)}+\frac{1}{||\Theta_q^{(k)}||^2}\Theta_q^{(k)}\otimes \Theta_q ^{(k)}\in H_k\otimes H_k. 
\end{equation}
It is entangled. The two theorems below are for the cases when $M$ is compact and when $M$ is ${\mathbb{C}}^n$ ($n\in{\mathbb{N}}$), respectively.  
We address the question how the entanglement entropy $E_k$ \cite{alieb} of 
\begin{equation}
\label{stateb}
b_k=b_k(p,q)=\frac{1}{||w_k||}w_k\in H_k\otimes H_k
\end{equation}
depends on the quantum parameter $k$ and on the distance between $p$ and $q$, ${\mathrm{dist}}(p,q)$. The theorems provide positive lower bounds  on $E_k(b_k)$. In quantum information theory, when quantum systems are used for communication, one interpretation of entanglement entropy is the amount of information that can be transmitted. Bell states are maximally entangled (e.g.  
\begin{equation}
\label{bellstate}
\frac{1}{\sqrt{2}}(e_0\otimes e_0+e_1\otimes e_1)
\end{equation}
is one of the standard  Bell states). 

Our state  (\ref{statew}) is constructed from the coherent vectors $\Theta_p^{(k)}$ and $\Theta_q^{(k)}$ which are typically not orthogonal to each other (unlike 
$e_0$ and $e_1$ in the Bell state (\ref{bellstate})), although $\langle \Theta_p^{(k)},\Theta_q^{(k)}\rangle \to 0$ as $k\to \infty$.   

To provide some background, in quantum information theory, given a nonzero vector $v$ in the tensor product of two separable Hilbert spaces $V_1$ and $V_2$, its entanglement entropy $E(v)$  characterizes "how nondecomposable" (or, in other words, how entangled) this vector is. It is defined as follows. For the purposes of this paper we only need the case $V_1=V_2=V$. Let $\langle .,.\rangle$ be the inner product in $V$. 
Let  $\{ e_i\}$ be an orthonormal basis in $V$. Let $Tr_2(A)\in End(V)$ denote the partial trace of a density matrix $A$. It is 
defined by 
$$
\langle x, Tr_2(A)y\rangle=\sum_i \langle x\otimes e_i, A(y\otimes e_i)\rangle 
$$ 
for every $x,y\in V$. The entanglement entropy is 
$$
E(v)=-Tr(\rho\ln \rho)=-\sum_i\langle (\rho \ln \rho ) e_i, e_i\rangle 
$$ 
where 
$$
\rho=Tr_2 (P_v)
$$
and $P_v$ is the rank $1$ orthogonal projection onto the $1$-dimensional linear subspace of $V\otimes V$ spanned by $v$. The operator $\rho\ln\rho$ is defined via the continuous functional calculus.  When $V$ is finite-dimensional, $E(v)$ is a real number in the interval $[0, \ln dim (V)]$. When $V$ is infinite-dimensional, $E(v)$ is a nonnegative real number or $+\infty$. The value of $E(v)$ does not depend on the choice of the basis $\{ e_i\}$.  

K\"ahler quantization deals with asymptotic analysis on K\"ahler manifolds in the context of classical mechanics and quantum mechanics. Let 
$(M,\omega)$ be an integral K\"ahler manifold. Let  $L\to M$ be a holomorphic line bundle whose first Chern class is represented by $\omega$. One can consider $(M,\omega)$ as a classical phase space, i.e. a space that parametrizes position and momentum of a classical particle. Classical mechanics on $M$ is captured in $\omega$ and a choice of a Hamiltonian (a smooth function on $M$). The symplectic form defines a Poisson bracket on $C^{\infty}(M)$. Dirac's correspondence principle  seeks a linear map from $C^{\infty}(M)$ to linear operators on the Hilbert space of quantum mechanical wave functions that takes the Poisson bracket of functions to the commutator of operators. In geometric quantization or K\"ahler quantization, a standard choice of the Hilbert space $V$ is the space of holomorphic sections of $L^k$, where the positive integer $k$ is interpreted (philosophically) as $1/\hbar$. If $M$ is compact, then $V$ is finite-dimensional.   If $M$ is noncompact, then $V$ is infinite-dimensional.   

The motivation to bring techniques from quantum information theory to geometric quantization was to obtain new insights in the interplay between geometry and analysis on K\"ahler manifolds. It would be interesting to investigate if there is a meaningful relationship between the information-theoretic entropy and other concepts of entropy. In the opposite direction, some geometric intuition may be useful in information transmission problems.  

\subsection{Compact case} 
\label{seccompact}
Let $L\to M$ be a positive hermitian holomorphic line bundle on a compact $n$-dimensional complex manifold $M$.  
Denote by $\nabla$ the Chern connection in $L$. Equipped with 
the $2$-form $\omega=i \ {\mathrm{curv}}(\nabla)$, $M$ is a K\"ahler manifold. Denote by $d\mu$ the measure on $M$ associated with the volume form 
$\dfrac{\omega^n}{n!}$. As before, let $k$ be a positive integer. For $p$, $q$ in $M$, let $\Theta_p^{(k)},\Theta_q^{(k)}\in H_k=H^0(M,L^k)$ be Rawnsley coherent states at $p$ and $q$ (see e.g. \cite{berss}). 

Let us recall the definition of $\Theta_p^{(k)}$ for $p\in M$ and $k\in{\mathbb{N}}$. Choose a unit vector $\xi\in L$. Then by Riesz representation theorem there is a unique vector $\Theta_p^{(k)}$ in the Hilbert space $H^0(M,L^k)$ with the property 
$$ 
\langle s,\Theta_p^{(k)}\rangle=\langle s(p),\xi^{\otimes k}\rangle 
$$
for every $s\in H^0(M,L^k)$. 

The $k\to\infty$ asymptotics of the norms  $||\Theta_p^{(k)}||$, $||\Theta_q^{(k)}||$ are determined by the asymptotics of the Bergman kernels for $L^k$ on the diagonal. We take these asymptotics from \cite{zeld}. Asymptotic bounds for the inner products $\langle \Theta_p^{(k)},\Theta_q^{(k)}\rangle$ 
can be obtained from the off-diagonal estimates on the Bergman kernels \cite{mamar}.    
\begin{theorem}
\label{thcompact}Suppose  $L\to M$ is a positive hermitian holomorphic line bundle on a compact $n$-dimensional complex manifold $M$.  
Let $p,q\in M$. Let $k\in {\mathbb{N}}$. Then there is the following (positive) lower bound for the entanglement entropy of the pure state $b_k(p,q)$ (\ref{stateb}). 
There are positive constants $C_1$ and $C_2$ that depend on $M$ and $\omega$ such that
as $k\to \infty$ 
$$
E_k(b_k)\ge \frac{1}{2}(1-C_1e^{-C_2\sqrt{k} \ {\mathrm{dist}}(p,q) } )^4.
$$ 
\end{theorem}

\subsection{$M={\mathbb{C}}^n$} 
Let $M={\mathbb{C}}^n$, $n\ge 1$. We will use the notations 
$$
z^T\bar{w}=z_1\bar w_1+...+z_n\bar w_n
$$
or 
$$
\langle z,w\rangle =z_1\bar w_1+...+z_n\bar w_n
$$
and 
$$
|z|=\sqrt{ z^T\bar{z}}
$$
or
$$
||z||=\sqrt{ z^T\bar{z}}
$$
for $z,w\in {\mathbb{C}}^n$. 

For $k\in {\mathbb{N}}$ let $H_k$ be the Segal-Bargmann space that consists of holomorphic functions on $M$ with the inner product 
$$
\langle f,g\rangle =\left (\frac{k}{\pi}\right )^n\int_Mf(z)\overline{g(z)}e^{-k|z|^2}dV(z), 
$$
where $dV$ is the Lebesgue measure on ${\mathbb{R}}^{2n}$. It is a reproducing kernel Hilbert space. For $p\in {\mathbb{C}}^n$ the coherent vector at $p$  is 
$\Theta_p^{(k)}\in H_k$ defined by 
\begin{equation}
\label{tetap}
\Theta_p^{(k)}(z)=e^{kz^T\bar p }.
\end{equation}
It is defined by the property 
\begin{equation}
\label{reprod}
\langle f,\Theta_p^{(k)}\rangle=f(p)
\end{equation}
for every $f\in H_k$. Similarly for $q\in M$ the coherent vector at $q$ is
\begin{equation}
\label{tetaq}
\Theta_q^{(k)}(z)=e^{kz^T\bar q }.
\end{equation}
\begin{theorem}
\label{thcn}
Let $p,q\in {\mathbb{C}}^n$. Let $k\in {\mathbb{N}}$. 
Let $\Theta_p^{(k)}$ and $\Theta_p^{(k)}$ be the coherent states at $p$ and $q$ respectively (\ref{tetap}), (\ref{tetaq}).   
Then there is the following (positive) lower bound for the entanglement entropy of the pure state $b_k(p,q)$ (\ref{stateb}): 
as $k\to \infty$ 
$$
E_k(b_k)\ge \frac{1}{2} (1-e^{-k |p-q| ^2} )^4.
$$ 
\end{theorem}
\begin{remark}
The change from $e^{-C\sqrt{k} \ {\mathrm{dist}}(p,q) }$ to $e^{-Ck \ {\mathrm{dist}}^2(p,q) }$ in Theorem \ref{thcn} reflects the fact 
that the latter appears in the Bergman asymptotics for real analytic metrics (see the discussion in \cite{hezari}). In the proof of Theorem \ref{thcompact} we used the Bergman kernel expansion for smooth metrics.     
\end{remark}

\section{Proofs}
\subsection{General lower bound} 
\begin{theorem}
\label{entropyth} Let $H$ be a separable Hilbert space, with the inner product $\langle .,.\rangle$. Let $u$ and $v$ be nonzero vectors in $H$, such that $u$ is not a multiple of $v$. Let
$$
w=\frac{1}{||u||^2}u\otimes u+\frac{1}{||v||^2}v\otimes v \in H\otimes H. 
$$
There is the following (positive) lower bound on the entanglement entropy $E(b)$ of the vector 
$b=\frac{1}{||w||}w$
\begin{equation}
\label{bound} 
E(b)\ge 2\frac{( ||u||^2||v||^2-|\langle u,v\rangle |^2)^2}{( 2 ||u||^2||v||^2+\langle u,v\rangle ^2+\langle v,u\rangle^2)^2}. 
\end{equation}
\end{theorem} 
\begin{proof} 
Let $e_1, e_2$ be an orthonormal basis of the $2$-dimensional complex linear subspace spanned by $u$ and $v$,  defined as follows: 
$$
e_1=\frac{1}{||u||}u,
$$
$e_2$ is the unit vector in the direction of $v-\langle v,e_1\rangle e_1$
$$
e_2=\frac{1}{||v-\langle v,e_1\rangle e_1||} (v-\langle v,e_1\rangle e_1)=\frac{1}{\sqrt{||v||^2-|\langle v,e_1\rangle |^2 }} (v-\langle v,e_1\rangle e_1)
$$
We get: 
$$
w=\frac{||v||^2+\langle v,e_1\rangle ^2}{||v||^2}e_1\otimes e_1+
\frac{\langle v,e_1\rangle  \sqrt{||v||^2-|\langle v,e_1\rangle |^2 }   }{||v||^2}(e_1\otimes e_2+e_2\otimes e_1)
$$
$$
+\frac { ||v||^2-|\langle v,e_1\rangle |^2 } { ||v||^2} e_2\otimes e_2
 $$
 $$
 ||w||=\frac{\sqrt{ 2||v||^2+ \langle v,e_1\rangle^2 +\langle e_1,v\rangle^2 } }{||v||} 
 $$
 Denote 
 $$
 \beta=\sqrt{ 2||v||^2+ \langle v,e_1\rangle^2 +\langle e_1,v\rangle^2 }.
 $$
 We get: 
 $$
b=\frac{1}{\beta ||v|| } 
\Bigl ( (||v||^2+\langle v,e_1\rangle ^2)e_1\otimes e_1+
$$
$$
\langle v,e_1\rangle  \sqrt{||v||^2-|\langle v,e_1\rangle |^2 }   (e_1\otimes e_2+e_2\otimes e_1)
+( ||v||^2-|\langle v,e_1\rangle |^2 )e_2\otimes e_2\Bigr )
 $$
 Let
 $$
 A=
 \frac{1}{\beta ||v|| } 
 \begin{pmatrix}
 ||v||^2+\langle v,e_1\rangle ^2 & \langle v,e_1\rangle  \sqrt{||v||^2-|\langle v,e_1\rangle |^2  }   \\
 \langle v,e_1\rangle  \sqrt{||v||^2-|\langle v,e_1\rangle |^2  } &
 ||v||^2-|\langle v,e_1\rangle |^2 
 \end{pmatrix}.
 $$
 Then $A^*A=$
 $$
 \frac{1}{\beta^2}
 \begin{pmatrix}
 ||v||^2+\langle v,e_1\rangle ^2+ \langle e_1,v\rangle ^2+|\langle v,e_1\rangle |^2 & 
  \sqrt{||v||^2-|\langle v,e_1\rangle |^2  } ( \langle v,e_1\rangle +\langle e_1,v\rangle) \\ 
 \sqrt{||v||^2-|\langle v,e_1\rangle |^2  } ( \langle v,e_1\rangle +\langle e_1,v\rangle) &
 ||v||^2-|\langle v,e_1\rangle |^2
 \end{pmatrix}
 $$
 The equation for the eigenvalues of $A^*A$ is 
 $$
 \lambda^2-\lambda+\frac{(||v||^2-|\langle v,e_1\rangle |^2 )^2}{\beta^4}=0.
 $$
 The eigenvalues of $A^*A$ are 
 $$
 \lambda_{1,2}=\frac{1}{2}\pm\frac{1}{2}\frac{(\langle v,e_1\rangle+ \langle e_1,v\rangle) \sqrt { 4||v||^2+ (\langle v,e_1\rangle- \langle e_1,v\rangle) ^2}    }{\beta^2}=
 $$
 \begin{equation}
 \label{eigenv}
 \frac{1}{2}\pm\frac{1}{2}\frac{(\langle u,v\rangle+ \langle v,u\rangle) \sqrt { 4||u||^2||v||^2+ (\langle u,v\rangle- \langle v,u\rangle) ^2}    }
 {2||u||^2||v||^2+ \langle u,v\rangle^2 +\langle v,u\rangle^2}.
 \end{equation}
 The singular values of $A$ are the square roots of the eigenvalues of $A^*A$. The entanglement entropy of $b$ equals 
 \begin{equation}
 \label{e(b)}
 E(b)=-\lambda_1 \ln \lambda_ 1-\lambda_2\ln \lambda_2.
\end{equation}
Since for $0<x<1$ 
$$
-\ln x> 1-x,
$$
we have: 
\begin{equation}
\label{ineq}
E(b)\ge \lambda_1 (1-\lambda_1)+\lambda_2(1- \lambda_2)=1-\lambda_1^2-\lambda_2^2.
\end{equation}
Now, (\ref{bound}) is obtained by plugging (\ref{eigenv}) into (\ref{ineq}). \end{proof} 

\begin{remark}
 Since $u$ and $v$ in Theorem \ref{entropyth} are linearly independent, it follows that  the vector $b$ is entangled, i.e. $E(b)>0$. 
 This follows from the fact that the right hand side of the inequality (\ref{bound}) is positive. Another way to see it is to refer to (\ref{e(b)}) and to observe that in (\ref{eigenv}) 
 $$
\Bigl |  \frac{(\langle u,v\rangle+ \langle v,u\rangle) \sqrt { 4||u||^2||v||^2+ (\langle u,v\rangle- \langle v,u\rangle) ^2}    }
 {2||u||^2||v||^2+ \langle u,v\rangle^2 +\langle v,u\rangle^2}\Bigr | <1
 $$
 (it is straightforward to check that this inequality is equivalent to  
 $$
 (||u||^2||v||^2- \langle u,v\rangle \langle v,u\rangle) ^2 >0),
 $$ 
 hence 
 $0<\lambda_1<1$ and $0<\lambda_2<1$. 
\end{remark}

\subsection{Proof of Theorem \ref{thcompact}}
\begin{proof}
It follows from \cite{zeld} that there is a constant $A_0>0$ such that  as $k\to\infty$, $||\Theta_p^{(k)}||^2$ and $||\Theta_q^{(k)}||^2$ are asymptotic to 
\begin{equation}
\label{asymptnorms}
A_0k^n+O(k^{n-1}).
\end{equation}
It follows from \cite{mamar} that there are constants $A_1>0$, $A_2>0$ such that  as $k\to\infty$ 
$$
|\langle \Theta_p^{(k)},\Theta_q^{(k)}\rangle|\le A_1k^ne^{-A_2\sqrt{k} \ {\mathrm{dist}}(p,q) }.
$$
From (\ref{bound}) in Theorem \ref{entropyth} we get: 
$$
E_k(b_k)\ge \frac{1}{2}\frac{(1-\frac{|\langle \Theta_p^{(k)},\Theta_q^{(k)}\rangle |^2}{  ||\Theta_p^{(k)}||^2||\Theta_q^{(k)}||^2}  )^2}{( 
1 +\frac{\langle \Theta_p^{(k)},\Theta_q^{(k)}\rangle ^2+\langle \Theta_q^{(k)},\Theta_p^{(k)}\rangle^2} 
{2 ||\Theta_p^{(k)}||^2||\Theta_q^{(k)}||^2} )^2}. 
$$
As $k\to\infty$, 
$$
E_k(b_k)\ge \frac{1}{2}(1-\frac{|\langle \Theta_p^{(k)},\Theta_q^{(k)}\rangle |^2}{  ||\Theta_p^{(k)}||^2||\Theta_q^{(k)}||^2}  )^2
( 1 -\frac{\langle \Theta_p^{(k)},\Theta_q^{(k)}\rangle ^2+\langle \Theta_q^{(k)},\Theta_p^{(k)}\rangle^2} 
{2 ||\Theta_p^{(k)}||^2||\Theta_q^{(k)}||^2} )^2\ge 
$$
$$
\frac{1}{2}(1-\frac{A_1^2k^{2n}e^{-2A_2\sqrt{k} \ {\mathrm{dist}}(p,q)}}{  ||\Theta_p^{(k)}||^2||\Theta_q^{(k)}||^2}  )^4.
$$
The conclusion now follows from (\ref{asymptnorms}). 
\end{proof} 

\subsection{Proof of Theorem \ref{thcn}}
\begin{proof}
Using (\ref{bound}) in Theorem \ref{entropyth} we get: 
$$
E_k(b_k)\ge 2\frac{( ||\Theta_p^{(k)}||^2||\Theta_q^{(k)}||^2-|\langle \Theta_p^{(k)},\Theta_q^{(k)}\rangle |^2)^2}
{( 2 ||\Theta_p^{(k)}||^2||\Theta_q^{(k)}||^2+\langle \Theta_p^{(k)},\Theta_q^{(k)}\rangle ^2+\langle \Theta_q^{(k)},\Theta_p^{(k)}\rangle^2)^2}.
$$
By the reproducing property (\ref{reprod}) 
$$
 ||\Theta_p^{(k)}||^2=\Theta_p(p)=e^{k|p|^2}
 $$
 $$
 ||\Theta_q^{(k)}||^2=\Theta_q(q)=e^{k|q|^2}
 $$
 $$
 \langle \Theta_p^{(k)},\Theta_q^{(k)}\rangle=\Theta_p^{(k)}(q)=e^{kq^T\bar p}.
 $$
 Therefore
 $$
 E_k(b_k)\ge \frac{1}{2}\frac{ (1- e^{kq^T\bar p +kp^T\bar q-k|p|^2-k|q|^2} )^2 }{(1+\frac{e^{2kq^T\bar p}+e^{2kp^T\bar q}}{2e^{k|p|^2+k|q|^2}})^2}.
 $$
 As $k\to\infty$
$$
E_k(b_k)\ge \frac{1}{2} (1- e^{kq^T\bar p +kp^T\bar q-k|p|^2-k|q|^2} )^2 (1-\frac{e^{2kq^T\bar p}+e^{2kp^T\bar q}}{2e^{k|p|^2+k|q|^2}})^2.
 $$
and the conclusion follows. 
\end{proof} 

\section{Example for \ref{seccompact}} Let $M={\mathbb{CP}}^1$ and let $L\to M$ be the hyperplane bundle with the standard hermitian metric. The K\"ahler form is the Fubini-Study form on $M$. The Hilbert space $V=H^0(M,L)$ is isomorphic to the space of polynomials in $1$ and $z$, with the inner product 
$$
\langle f,g\rangle =\frac{2}{\pi}\int_{{\mathbb{C}}}\frac{f(z)\overline{g(z)}}{(1+|z|^2)^3}dx \ dy   .
$$
The monomials $e_0=1$ and $e_1=z$ form an orthonormal basis in $V$. 
For $p\in{\mathbb{C}}$ (an affine chart of $M$)  let $\Theta_p$ be the unique vector in $V$ defined by the property 
$$
\langle f,\Theta_p\rangle=f(p) 
$$
for all $f\in V$. It is immediate that 
$$
\Theta_p(z)=1+z\bar p
$$
and
$$
||\Theta_p||=\sqrt{1+|p|^2}.
$$  
Let us consider two particles at $p=x+i0$ and $q=-x+i0$, $x>0$. The associated state (\ref{stateb}) is 
$$
b_1=b_1(p,q)=\frac{1}{\sqrt{1+x^4}}e_0\otimes e_0+\frac{x^2}{\sqrt{1+x^4}}e_1\otimes e_1.
$$  
The Schmidt coefficients are $\alpha_0=\frac{1}{\sqrt{1+x^4}}$ and $\alpha_1=\frac{x^2}{\sqrt{1+x^4}}$. The entanglement entropy of $b_1$ equals 
\begin{equation}
\label{entropyx}
E(x)=-\alpha_1^2 \ln(\alpha_1^2)-\alpha_2^2 \ln (\alpha_2^2)=
\frac{(1+x^4)\ln (1+x^4)-x^4\ln (x^4) } {1+x^4}. 
\end{equation}
The graph of $E$ is shown in Figure \ref{fig}.  
From (\ref{entropyx}), we observe that $E(x)\to 0$ as $x\to \infty$, and the maximum value of $E(x)$ is attained at $x=1$.  
\begin{figure}
\includegraphics[width=2.5in]{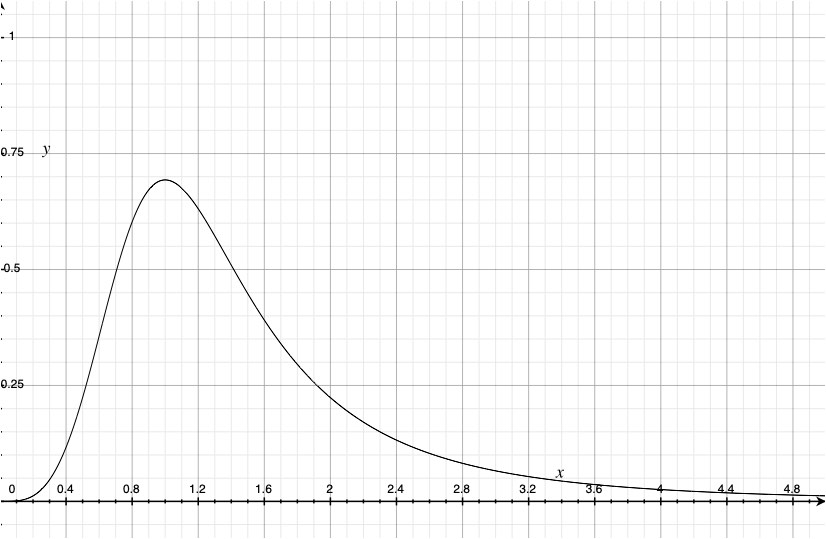}
\caption{The graph of $E(x)$, $x>0$.} \label{fig}
\end{figure} 



%
%
%
%

\end{document}